\documentclass[12pt]{article}
\usepackage{amsmath,latexsym,amssymb}
\usepackage{enumerate}
\usepackage{amsthm}

\newcommand\A{\mathcal{A}}
\newcommand\NN{\mathbb{N}}
\newcommand\RR{\mathbb{R}}

\newcommand\MM{\mathcal{M}}

\newcommand\DD{\mathcal{D}}

\newcommand\KK{\mathcal{K}}

\newcommand\UU{\mathcal{U}}

\newcommand\WW{\mathcal{W}}
\newcommand\YY{\mathcal{Y}}

\DeclareMathOperator\diam{diam}

\renewcommand{\int}{\operatorname{int}}

\newtheorem{theorem}{Theorem}[section]
\newtheorem{definition}[theorem]{Definition}
\newtheorem{corollary}[theorem]{Corollary}

\newtheorem{lemma}[theorem]{Lemma}
\newtheorem{remark}[theorem]{Remark}
\newtheorem{proposition}[theorem]{Proposition}

\newtheorem{question}[theorem]{Question}

\numberwithin{equation}{section}

\newcommand{\address}{Address: Department of Mathematics, University of North Texas, 1155 Union Circle \#311430, Denton, TX 76203-5017, USA; E-mail: allaart@unt.edu}

\title{On univoque and strongly univoque sets}
\author{Pieter C. Allaart \footnote{\address}}

\begin{document}

\maketitle

\begin{abstract}
Much has been written about expansions of real numbers in noninteger bases. Particularly, for a finite alphabet $\{0,1,\dots,\alpha\}$ and a real number (base) $1<\beta<\alpha+1$, the so-called {\em univoque set} of numbers which have a unique expansion in base $\beta$ has garnered a great deal of attention in recent years. Motivated by recent applications of $\beta$-expansions to Bernoulli convolutions and a certain class of self-affine functions, we introduce the notion of a {\em strongly univoque} set. We study in detail the set $\DD_\beta$ of numbers which are univoque but not strongly univoque. Our main result is that $\DD_\beta$ is nonempty if and only if the number $1$ has a unique nonterminating expansion in base $\beta$, and in that case, $\DD_\beta$ is uncountable. We give a sufficient condition for $\DD_\beta$ to have positive Hausdorff dimension, and show that, on the other hand, there are infinitely many values of $\beta$ for which $\DD_\beta$ is uncountable but of Hausdorff dimension zero.

\bigskip
{\it AMS 2010 subject classification}: 11A63 (primary),  28A78 (secondary)

\bigskip
{\it Key words and phrases}: Beta-expansion, univoque set, strongly univoque set, Hausdorff dimension, $p$-mirror sequence.
\end{abstract}

\section{Introduction and main results}

In the last 25 years or so, there has been significant interest in expansions of real numbers in noninteger bases; see \cite{Komornik} for an excellent survey of the pre-2011 literature. In the general setting considered here, we have a finite alphabet $A:=\{0,1,\dots,\alpha\}$, where $\alpha\in\NN$ is fixed, and a real number $\beta$ with $1<\beta<\alpha+1$. For $x\in\RR$, we call an expression of the form
\begin{equation}
x=\sum_{j=1}^\infty \frac{\omega_j}{\beta^j}, \qquad\mbox{where $\omega_1,\omega_2,\dots \in A$}
\label{eq:beta-expansion}
\end{equation}
an {\em expansion of $x$ in base $\beta$} (or simply, a {\em $\beta$-expansion}). By a classical result of R\'enyi \cite{Renyi} such an expansion of $x$ exists if and only if $x\in J_\beta':=[0,\alpha/(\beta-1)]$. Observe that we allow for the possibility that $\beta\leq\alpha$, in which case digits larger than the base $\beta$ may be used. It is well known (see \cite{Sidorov}) that almost every $x\in J_\beta'$ has a continuum of $\beta$-expansions, but there are also many points having a unique expansion. These include, trivially, the endpoints $0$ and $\alpha/(\beta-1)$ of $J_\beta'$. We will exclude these points from consideration here and define the set
\begin{equation}
\A_\beta':=\{x\in \int(J_\beta'): x\ \mbox{has a {\em unique} expansion of the form \eqref{eq:beta-expansion}}\},
\label{eq:univoque-set}
\end{equation}
which has been called a {\em univoque} set and is known to have a complicated structure. Specifically (see \cite{GlenSid,KongLi,KLD}), there is a number $G:=G(\alpha)>1$ (introduced in \cite{Baker} and called a {\em generalized golden ratio}) and a critical value $\beta_c:=\beta_c(\alpha)\in(G,\alpha+1)$, called the {\em Komornik-Loreti constant}, such that $\A_\beta'$ is empty for $\beta\leq G$; nonempty but countable for $G<\beta<\beta_c$; uncountable but with Hausdorff dimension zero for $\beta=\beta_c$; and of positive Hausdorff dimension for $\beta>\beta_c$. The Komornik-Loreti constant was introduced in the papers \cite{KomLor1,KomLor2}.

The set $\A_\beta'$ is symmetric (i.e. invariant under the map $x\mapsto \bar{x}:=\frac{\alpha}{\beta-1}-x$), as can be seen by replacing each $\omega_j$ with $\alpha-\omega_j$ in \eqref{eq:beta-expansion}. We will be interested in a subset of $\A_\beta'$ defined by a slightly stronger condition, explained below, which is motivated by some recent applications of $\beta$-expansions. For ease of presentation, we initially restrict ourselves to the interval $J_\beta:=[(\alpha-\beta+1)/(\beta-1),1]$, and write $\A_\beta:=\A_\beta'\cap \int(J_\beta)$. Note that this set is again invariant under the map $x\mapsto \bar{x}$. Later we will show that all our results hold also when $\A_\beta$ is replaced with the larger set $\A_\beta'$.

Observe that the interval $J_\beta$ is nonempty if and only if $\beta\geq 1+\alpha/2$. However, as shown in \cite{Baker}, $G(\alpha)\geq 1+\alpha/2$ for each $\alpha$, and therefore $\A_\beta'=\emptyset$ when $\beta\leq 1+\alpha/2$, so this range of bases is not of interest for us here.

The set $\A_\beta$ is most easily studied by considering the set of all corresponding sequences $\omega_1\omega_2\cdots$. 
Let $\Omega:=A^\NN$. For $1<\beta<\alpha+1$, let $\Pi_\beta:\Omega\to\RR$ denote the projection map given by
\begin{equation*}
\Pi_\beta(\omega)=\sum_{j=1}^\infty \frac{\omega_j}{\beta^j}, \qquad \omega=\omega_1\omega_2\cdots \in\Omega,
\end{equation*}
so that \eqref{eq:beta-expansion} can be written compactly as $x=\Pi_\beta(\omega)$. Define the set
\begin{equation}
\UU_\beta:=\Pi_\beta^{-1}(\A_\beta).
\label{eq:definition-of-U}
\end{equation}
Let $\sigma$ denote the left shift map on $\Omega$; that is, $\sigma(\omega_1\omega_2\cdots)=\omega_2\omega_3\cdots$. For a number $a\in A$, let $\bar{a}:=\alpha-a$ denote the reflection of $a$. For $\omega=\omega_1\omega_2\cdots\in\Omega$, likewise denote $\bar{\omega}=\bar{\omega}_1\bar{\omega}_2\cdots$, and similarly define the reflection for finite words.
It is a straightforward consequence of the definitions that
\begin{equation}
\omega\in\UU_\beta \qquad \Longleftrightarrow \qquad \Pi_\beta(\sigma^n(\omega))<1\ \mbox{and}\ \Pi_\beta(\sigma^n(\bar{\omega}))<1\ \mbox{for all $n\geq 0$}.
\end{equation}
(See \cite{DK1} or \cite[Lemma 5.1]{Allaart} for the case $\alpha=1$.)
In some recent applications of $\beta$-expansions \cite{JSS,Allaart}, it was necessary to consider the (conceivably smaller) set
\begin{equation}
\widetilde{\UU}_\beta:=\{\omega\in\UU_\beta: \limsup_{n\to\infty}\Pi_\beta(\sigma^n(\omega))<1\ \mbox{and}\ \limsup_{n\to\infty}\Pi_\beta(\sigma^n(\bar{\omega}))<1\}.
\label{eq:alternative-representation-of-U-tilde}
\end{equation}
We shall call elements of $\widetilde{\UU}_\beta$ and their projections under $\Pi_\beta$ {\em strongly univoque}, and refer to $\widetilde{\UU}_\beta$ and its projection $\Pi_\beta(\widetilde{\UU}_\beta)\subset \A_\beta$ as {\em strongly univoque sets}. In \cite{JSS}, the sets $\widetilde{\UU}_\beta$ were used to study the multifractal spectra of asymmetric Bernoulli convolutions. In \cite{Allaart}, the present author studied a one-parameter family $\{F_a: 0<a<1\}$ of self-affine functions introduced by Okamoto \cite{Okamoto}, and showed that for $a>1/2$, the set of points where $F_a$ has an infinite derivative is intimately connected to the sets $\UU_\beta$ and $\widetilde{\UU}_\beta$, where $\beta=1/a$. The characterization of these infinite derivatives is most elegant in the case where $\widetilde{\UU}_\beta=\UU_\beta$. 

Motivated by this second application, we define the sets
\begin{equation*}
\WW_\beta:=\UU_\beta\backslash \widetilde{\UU}_\beta, \qquad 1<\beta<\alpha+1,
\end{equation*}
and put $\DD_\beta:=\Pi_\beta(\WW_\beta)$. Thus, $\DD_\beta$ is the set of numbers in $J_\beta$ which are univoque but not strongly univoque.
We aim to characterize the values of $\beta$ for which $\DD_\beta$ is nonempty, and, if it is, to investigate its cardinality and Hausdorff dimension.
Since we will see in Section \ref{sec:prelim} that the restriction of the map $\Pi_\beta$ to $\UU_\beta$ is bi-Lipschitz continuous (with respect to the metric $\rho_\beta$ defined in \eqref{eq:metric} below), the sets $\DD_\beta$ and $\WW_\beta$ have the same cardinality and Hausdorff dimension. Thus, all statements given below about $\WW_\beta$ hold for $\DD_\beta$ as well.

Our first result, an easy consequence of known facts from the literature, says that for most $\beta$, all univoque numbers are in fact strongly univoque.

\begin{proposition} \label{prop:smallness}
The set $\{\beta\in(1,\alpha+1): \WW_\beta\neq\emptyset\}$ is nowhere dense and of Lebesgue measure zero. Moreover, $\WW_\beta=\emptyset$ when $\beta<\beta_c$.
\end{proposition}

In order to state the main theorems of this paper, some more notation is needed. First, we introduce the lexicographic ordering on $\Omega$: For $\omega=\omega_1\omega_2\cdots$ and $\eta=\eta_1\eta_2\cdots$ in $\Omega$, we write $\omega<\eta$ if there is an index $n\in\NN$ such that $\omega_i=\eta_i$ for $i=1,\dots,n-1$, and $\omega_n<\eta_n$. Likewise, for finite words $\omega=\omega_1\cdots\omega_m$ and $\eta=\eta_1\cdots\eta_m$ of the same length, we write $\omega<\eta$ if there is an index $n\leq m$ such that $\omega_i=\eta_i$ for $i=1,\dots,n-1$, and $\omega_n<\eta_n$. For finite or infinite sequences $\omega$ and $\eta$, we write $\omega\leq\eta$ if $\omega<\eta$ or $\omega=\eta$.

Call an expansion $\omega$ of a number $x$ {\em infinite} if $\omega_i>0$ for infinitely many $i$. The lexicographically largest infinite expansion of $x$ is called the {\em quasi-greedy expansion} of $x$. Quasi-greedy expansions were introduced formally in \cite{DK2} (where they are called quasi-regular expansions), but are already implicit in the work of Parry \cite{Parry}. A critical role in the study of $\beta$-expansions is played by the quasi-greedy expansion of $x=1$, which we shall denote by $d=d_1 d_2\cdots$. We shall often write $d=d(\beta)$ to express that $d$ is the quasi-greedy expansion of $1$ in base $\beta$. It is well known that $d$ satisfies the inequalities
\begin{equation*}
\sigma^n(d)\leq d, \qquad n\in\NN,
\end{equation*}
and unless $d$ is periodic, the inequality is strict for each $n$.
There is also a simple lexicographic characterization of $\UU_\beta$, namely
\begin{equation}
\omega\in\UU_\beta\quad\Longleftrightarrow \quad \sigma^n(\omega)<d\ \mbox{and}\ \sigma^n(\bar{\omega})<d\ \mbox{for all $n\geq 0$}.
\label{eq:U-beta-characterization}
\end{equation}
This is essentially due to Parry \cite{Parry}; see also \cite{EHJ,EJK,GlenSid,KongLi}. The corresponding characterization of the larger set $\Pi_\beta^{-1}(\A_\beta')$ is slightly more subtle; see \cite{KongLi,deVries} and \eqref{eq:bigger-U-characterization} below.

Some of our results below have an interpretation in terms of the set
\begin{equation}
\A(\alpha):=\{\beta\in(1,\alpha+1): 1\ \mbox{has a unique expansion in base}\ \beta\}.
\end{equation}
This set is akin to the set $\UU$ studied in \cite{KomLor3}, but we emphasize that $\A(\alpha)$ depends on the size of the alphabet $A$, which is fixed, whereas the set $\UU$ in \cite{KomLor3} consists of those bases $\beta$ for which $1$ has a unique expansion in base $\beta$ using only digits strictly smaller than $\beta$. Nonetheless, the sets $\UU$ and $\A(\alpha)$ have very similar properties. As usual, we denote by $\overline{\A(\alpha)}$ the topological closure of $\A(\alpha)$.

The following theorem is the main result of this paper. The equivalences (iii) $\Leftrightarrow$ (iv) and (iii) $\Leftrightarrow$ (v) were proved in \cite{deVries} and \cite{KomLor3}, respectively, but are included here to put our result in historical perspective.

\begin{theorem} \label{thm:when-non-empty}
For $1<\beta<1+\alpha$, the following statements are equivalent:
\begin{enumerate}[(i)]
\item $\WW_\beta\neq\emptyset$;
\item $\WW_\beta$ is uncountable;
\item $\sigma^n(\bar{d})<d$ for every $n\in\NN$, where $d=d(\beta)$;
\item The number $1$ has a unique infinite expansion in base $\beta$.
\item $\beta\in\overline{\A(\alpha)}$.
\end{enumerate}
\end{theorem}

It has been shown, in \cite{DK2} for the case $\alpha=1$ and in \cite{KKL} for the general case, that $\A(\alpha)$ has Hausdorff dimension 1. Thus, we immediately get the following consequence of Theorem \ref{thm:when-non-empty}, which the reader may contrast with Proposition \ref{prop:smallness}.

\begin{corollary} \label{cor:full-HD}
The set $\{\beta\in(1,\alpha+1): \WW_\beta\neq\emptyset\}$ has Hausdorff dimension 1.
\end{corollary}

In \cite{Allaart}, where $\alpha=1$ and hence $1<\beta<2$, it is shown that Okamoto's function $F_a$ with $a=1/\beta$ has an infinite derivative at a point $x$ if and only if 
\begin{equation}
\lim_{n\to\infty}(3/\beta)^n\big(1-\Pi_\beta(\sigma^n(\omega))\big)=\lim_{n\to\infty}(3/\beta)^n\big(1-\Pi_\beta(\sigma^n(\bar{\omega}))\big)=\infty,
\label{eq:both-going-to-infinity}
\end{equation}
where $\omega$ is derived from the ternary expansion of $x$. In the follow-up paper \cite{Allaart2} a sequence of functions is identified whose infinite derivatives correspond similarly to $\beta$-expansions over the more general alphabet considered here. 
The limits in \eqref{eq:both-going-to-infinity} hold automatically when $\omega\in\widetilde{\UU}_\beta$, clearly fail when $\omega\not\in\UU_\beta$, but may or may not hold when $\omega\in\WW_\beta=\UU_\beta\backslash\widetilde{\UU}_\beta$. This suggests that it is meaningful to study the rate of convergence to $1$ for subsequences of $\Pi_\beta(\sigma^n(\omega))$. As the next theorem indicates, any desired rate of convergence can be achieved by a suitable choice of $\omega$, provided of course that $\WW_\beta$ is nonempty.

\begin{theorem} \label{thm:rate-of-convergence}
\begin{enumerate}[(i)]
\item Assume $\WW_\beta\neq\emptyset$. Then for any sequence $(\theta_n)_{n\in\NN}$ of positive real numbers there exist uncountably many $\omega\in\WW_\beta$ such that
\begin{equation*}
\liminf_{n\to\infty} \theta_n \big(1-\Pi_\beta(\sigma^n(\omega))\big)<\infty.
\end{equation*}
\item For each $\beta\geq\beta_c$ and for any increasing sequence $(\theta_n)_{n\in\NN}$ of positive real numbers with $\lim_{n\to\infty}\theta_n=\infty$, there exist uncountably many $\omega\in\UU_\beta$ such that
\begin{equation}
\lim_{n\to\infty} \theta_n \big(1-\Pi_\beta(\sigma^n(\omega))\big)=\lim_{n\to\infty} \theta_n \big(1-\Pi_\beta(\sigma^n(\bar{\omega}))\big)=\infty.
\label{eq:both-infinite}
\end{equation}
\end{enumerate}
\end{theorem}

The set $\DD_\beta$, being a subset of $\A_\beta$, is Lebesgue null, and therefore it is interesting to ask about its Hausdorff dimension (or equivalently, that of $\WW_\beta$). In order to formulate precise results, we first introduce a metric $\rho_\beta$ on $\Omega$, defined by 
\begin{equation}
\rho_\beta(\omega,\eta):=\beta^{1-\inf\{n:\omega_n\neq \eta_n\}}.
\label{eq:metric}
\end{equation}
(The reason for the extra `$1$' is that this way, $\diam(\Omega)=1$.) Equipped with this metric and the lexicographic order ``$<$", $\Omega$ is a linearly ordered metric space, and so statements such as $\limsup_{n\to\infty} \sigma^n(\omega)=d$ are meaningful. Note by \eqref{eq:U-beta-characterization} that if $\omega\in\UU_\beta$, then $\limsup \sigma^n(\omega)=d$ if and only if for each $k\in\NN$, there is an index $n$ such that $\omega_{n+1}\cdots\omega_{n+k}=d_1\cdots d_k$. 

For a Borel set $E\subset \Omega$, we denote by $\dim_H E$ the Hausdorff dimension of $E$ induced by the metric $\rho_\beta$, where the value of $\beta$ will usually be clear form the context. Note that the value of $\dim_H E$ depends on the choice of $\beta$, but whether $\dim_H E$ is positive or zero is independent of $\beta$.

\begin{theorem} \label{thm:Hausdorff-dimension}
\begin{enumerate}[(i)]
\item $\dim_H \UU_\beta=\dim_H \widetilde{\UU}_\beta$ for every $\beta$.
\item Let $d=d(\beta)$. If $\limsup_{n\to\infty} \sigma^n(\bar{d})<d$, then $\dim_H \WW_\beta>0$.
\item There are infinitely many values of $\beta$ in $(1,\alpha+1)$ such that $\WW_\beta$ is uncountable but $\dim_H \WW_\beta=0$. The smallest such $\beta$ is $\beta_c$.
\end{enumerate}
\end{theorem}

The above theorem leaves several interesting questions.

\begin{question}
{\rm
Does the converse of (ii) hold? In other words, does 
\begin{equation}
\limsup_{n\to\infty} \sigma^n(\bar{d})=d
\label{eq:self-reflect}
\end{equation} 
imply that $\dim_H \WW_\beta=0$? (Note that if $\limsup_{n\to\infty} \sigma^n(\bar{d})>d$, then $\WW_\beta=\emptyset$ by Theorem \ref{thm:when-non-empty}.)
}
\end{question}

\begin{question}
{\rm
For which values of $\beta$ (other than those specified in Theorem \ref{thm:Hausdorff-dimension-p-mirror} below) does $d:=d(\beta)$ satisfy \eqref{eq:self-reflect}? 
}
\end{question}

\begin{question}
{\rm
One would expect that $\WW_\beta$ is always a small subset of $\UU_\beta$. Is it true that $\dim_H \WW_\beta<\dim_H \UU_\beta$ for all $\beta$?
}
\end{question}

\subsection{Extension to the interval $J_\beta'$}

We now show that all of the above results remain valid if the set $\A_\beta$ is replaced with the larger set $\A_\beta'$ defined in \eqref{eq:univoque-set}. By analogy with \eqref{eq:definition-of-U}, define the set
\begin{equation*}
\UU_\beta':=\Pi_\beta^{-1}(\A_\beta').
\end{equation*}
Instead of the simple characterization \eqref{eq:U-beta-characterization}, we have the following (see \cite[Theorem 3.4]{KongLi}):

\begin{lemma} \label{lem:bigger-U-characterization}
Let $\omega=(\omega_n)_n\in\Omega\backslash \{000\cdots,\alpha\alpha\alpha\cdots\}$, and let
\begin{equation*}
m^+(\omega):=\min\{n:\omega_n<\alpha\}, \qquad m^-(\omega):=\min\{n:\omega_n>0\}.
\end{equation*}
Then $\omega\in\UU_\beta'$ if and only if
\begin{equation}
\sigma^n(\omega)<d\ \mbox{for all $n\geq m^+(\omega)$} \quad\mbox{and} \quad \sigma^n(\bar{\omega})<d\ \mbox{for all $n\geq m^-(\omega)$},
\label{eq:bigger-U-characterization}
\end{equation}
where $d=d(\beta)$.
\end{lemma}

(A precursor of this lemma was proved in \cite{EJK}.) One checks easily that \eqref{eq:bigger-U-characterization} is equivalent to the condition
\begin{equation*}
\Pi_\beta(\sigma^n(\omega))<1\ \mbox{for all $n\geq m^+(\omega)$} \quad\mbox{and} \quad \Pi_\beta(\sigma^n(\bar{\omega}))<1\ \mbox{for all $n\geq m^-(\omega)$}.
\end{equation*}
It is therefore natural to define a set $\widetilde{U}_\beta'$ by complete analogy with the definition of $\widetilde{U}_\beta$, i.e.
\begin{equation*}
\widetilde{\UU}_\beta':=\{\omega\in\UU_\beta': \limsup_{n\to\infty}\Pi_\beta(\sigma^n(\omega))<1\ \mbox{and}\ \limsup_{n\to\infty}\Pi_\beta(\sigma^n(\bar{\omega}))<1\}.
\end{equation*}
Let $\WW_\beta':=\UU_\beta'\backslash \widetilde{\UU}_\beta'$. It follows immediately that $\WW_\beta\subset\WW_\beta'$. Vice versa, if $\omega\in\WW_\beta'$, then putting $N:=\max\{m^+(\omega),m^-(\omega)\}$ we have by \eqref{eq:U-beta-characterization} and \eqref{eq:bigger-U-characterization} that $\sigma^N(\omega)\in \UU_\beta$. But since $\omega\not\in\widetilde{\UU}_\beta'$, it follows from the definitions of $\widetilde{\UU}_\beta$ and $\widetilde{\UU}_\beta'$ that $\sigma^N(\omega)\not\in\widetilde{\UU}_\beta$, and so $\sigma^N(\omega)\in\WW_\beta$. As a result,
\begin{equation*}
\WW_\beta'\subset \bigcup_{N=1}^\infty \bigcup_{\eta_1\cdots \eta_N\in A^N} \{\eta_1\cdots \eta_N \omega: \omega\in\WW_\beta\}.
\end{equation*}
Since the right hand side is a countable union of similar copies of $\WW_\beta$, and by Theorem \ref{thm:when-non-empty} $\WW_\beta$ is either empty or uncountable, it follows that $\WW_\beta$ and $\WW_\beta'$ have the same cardinality and the same Hausdorff dimension. Therefore, we may replace $\WW_\beta$ with $\WW_\beta'$ in each of the above results without affecting their validity.

\bigskip

The remainder of this article is organized as follows. Section \ref{sec:prelim} reviews necessary background from the literature on $\beta$-expansions, including the important notion of a $p$-mirror sequence, and includes a proof of Proposition \ref{prop:smallness}. Theorems \ref{thm:when-non-empty} and \ref{thm:rate-of-convergence} are proved in Section \ref{sec:main-proof}, and Theorem \ref{thm:Hausdorff-dimension} is proved in Section \ref{sec:HD}.

\section{Preliminaries} \label{sec:prelim}

The first two lemmas below were proved in \cite{JSS} for the case $\alpha=1$. We include short proofs here for completeness.

\begin{lemma} \label{lem:containment}
If $\beta_1>\beta_2$, then $\widetilde{\UU}_{\beta_1}\supset \UU_{\beta_2}$.
\end{lemma}

\begin{proof}
It suffices to find $\delta>0$ such that for $\omega\in\Omega$ the implication
\begin{equation}
\Pi_{\beta_2}(\omega)<1 \qquad \Longrightarrow \qquad \Pi_{\beta_1}(\omega)\leq 1-\delta
\label{eq:order-implication}
\end{equation}
holds; applying this to $\sigma^n(\omega)$ and $\sigma^n(\bar{\omega})$ in place of $\omega$ then gives the lemma.

Choose $M\in\NN$ large enough so that
$\alpha\beta_1^{-M}<\beta_1-1$,
and let
\begin{equation*}
\delta:=\min\left\{1-\frac{\alpha\beta_1^{-M}}{\beta_1-1},\min_{1\leq j\leq M}(\beta_2^{-j}-\beta_1^{-j})\right\}.
\end{equation*}
If $\omega_j=0$ for $j=1,\dots,M$, then
\begin{equation*}
\Pi_{\beta_1}(\omega)=\sum_{j=M+1}^\infty \frac{\omega_j}{\beta_1^j}\leq \frac{\alpha\beta_1^{-M}}{\beta_1-1}\leq 1-\delta.
\end{equation*}
On the other hand, if $\omega_{j_0}>0$ for some $j_0\in\{1,\dots,M\}$, then
\begin{equation*}
\Pi_{\beta_2}(\omega)-\Pi_{\beta_1}(\omega) \geq \omega_{j_0}(\beta_2^{-j_0}-\beta_1^{-j_0})\geq \min_{1\leq j\leq M}(\beta_2^{-j}-\beta_1^{-j})\geq \delta.
\end{equation*}
In both cases, \eqref{eq:order-implication} follows.
\end{proof}

\begin{lemma} \label{lem:bi-Lipschitz}
The map $\Pi_\beta|_{\UU_\beta}$ is bi-Lipschitz from $(\UU_\beta,\rho_\beta)$ onto $\Pi_\beta(\UU_\beta)$.
\end{lemma}

\begin{proof}
Let $\omega=\omega_1\omega_2\cdots$ and $\eta=\eta_1\eta_2\cdots$ be in $\UU_\beta$. If $\min\{n:\omega_n\neq\eta_n\}=k+1$, then $\rho_\beta(\omega,\eta)=\beta^{-k}$, and
\begin{equation*}
|\Pi_\beta(\omega)-\Pi_\beta(\eta)|=\beta^{-k}\big|\Pi_\beta(\sigma^k(\omega))-\Pi_\beta(\sigma^k(\eta))\big|.
\end{equation*}
Hence, it suffices to show that there are constants $0<C_1<C_2$ such that
\begin{equation}
\omega_1\neq\eta_1 \qquad\Longrightarrow \qquad C_1\leq|\Pi_\beta(\omega)-\Pi_\beta(\eta)|\leq C_2.
\label{eq:bi-Lipschitz}
\end{equation}
Assume without loss of generality that $\omega_1>\eta_1$. On the one hand,
\begin{equation*}
|\Pi_\beta(\omega)-\Pi_\beta(\eta)|\leq \sum_{j=1}^\infty\frac{|\omega_j-\eta_j|}{\beta^j}\leq \sum_{j=1}^\infty \frac{\alpha}{\beta^j}
=\frac{\alpha}{\beta-1}=:C_2.
\end{equation*}
On the other hand, $\omega\in\UU_\beta$ implies $\Pi_\beta(\sigma(\bar{\omega}))<1$, so $\sum_{j=2}^\infty \bar{\omega}_j\beta^{-j}<\beta^{-1}$, and hence,
\begin{equation*}
\sum_{j=2}^\infty \frac{\omega_j}{\beta^j}=\sum_{j=2}^\infty\frac{\alpha-\bar{\omega}_j}{\beta^j}
>\sum_{j=2}^\infty \frac{\alpha}{\beta^j}-\frac{1}{\beta}=\frac{1}{\beta}\left(\frac{\alpha+1-\beta}{\beta-1}\right),
\end{equation*}
while $\eta\in\UU_\beta$ implies $\Pi_\beta(\sigma(\eta))<1$, so $\sum_{j=2}^\infty \eta_j\beta^{-j}<\beta^{-1}$. Thus,
\begin{align*}
\Pi_\beta(\omega)-\Pi_\beta(\eta)&=\frac{\omega_1-\eta_1}{\beta}+\sum_{j=2}^\infty\frac{\omega_j-\eta_j}{\beta^j}\\
&>\frac{1}{\beta}+\frac{1}{\beta}\left(\frac{\alpha+1-\beta}{\beta-1}-1\right)=\frac{\alpha+1-\beta}{\beta(\beta-1)}=:C_1.
\end{align*}
Hence, we have \eqref{eq:bi-Lipschitz}.
\end{proof}

\begin{lemma}[\cite{deVries}, Proposition 2.4] \label{lem:increasing-projection}
The projection map $\Pi_\beta$ is strictly increasing on $\UU_\beta$.
\end{lemma}

In view of the last two lemmas, any subset of $\UU_\beta$ has the same Hausdorff dimension and cardinality as its projection under $\Pi_\beta$, and the restriction of $\Pi_\beta$ to $\UU_\beta$ is an order-preserving homeomorphism. This allows us to work entirely within the symbol space $\Omega$, as illustrated by the following important characterization.

\begin{proposition} \label{prop:translation}
Let $d=d(\beta)$. For $\omega\in\UU_\beta$, we have
\begin{equation*}
\limsup_{n\to\infty}\Pi_\beta(\sigma^n(\omega))=1 \qquad \Longleftrightarrow \qquad \limsup_{n\to\infty} \sigma^n(\omega)=d.
\end{equation*}
\end{proposition}

\begin{proof}
This is a direct consequence of Lemmas \ref{lem:bi-Lipschitz} and \ref{lem:increasing-projection}.
\end{proof}

\subsection{De Vries-Komornik numbers and $p$-mirror sequences}

An important role in the theory of $\beta$-expansions is played by a certain countably infinite set of bases, which we now introduce. For $a\in A$ with $a<\alpha$, we denote $a+1$ more compactly by $a^+$; and for $a\in A$ with $a>0$ we denote $a-1$ by $a^-$. For a word $t_1\cdots t_n$ in $A^n$, let $t_1\cdots t_n^+:=t_1\cdots t_{n-1}t_n^+$ if $t_n<\alpha$, and $t_1\cdots t_n^-:=t_1\cdots t_{n-1}t_n^-$ if $t_n>0$. The following definition is taken from \cite{KongLi}, with a minor modification -- see Remark \ref{rem:p-mirror} below. 

\begin{definition} \label{def:admissible-block}
{\rm
Let $p\in\NN$. A word $t_1\cdots t_p$ in $A^p$ is {\em admissible} if $t_p<\alpha$ and for each $1\leq i\leq p$ we have
\begin{equation}
\overline{t_i\cdots t_p t_1\cdots t_{i-1}}\leq t_1\cdots t_p^+ \qquad \mbox{and} \qquad t_i\cdots t_p^+\overline{t_1\cdots t_{i-1}}\leq t_1\cdots t_p^+.
\label{eq:admissible-block}
\end{equation}
}
\end{definition}

\begin{definition} \label{def:p-mirror-sequence}
{\rm
A sequence $d=(d_i)$ in $\Omega$ is called a {\em $p$-mirror sequence} if there is $p\in\NN$ and an admissible word $t_1\cdots t_p$ such that
\begin{equation*}
d_1\cdots d_p=t_1\cdots t_p^+,
\end{equation*}
and
\begin{equation*}
d_{2^m p+1}\cdots d_{2^{m+1}p}=\overline{d_1\cdots d_{2^m p}}^{\,+}, \qquad \mbox{for all $m\geq 0$}.
\end{equation*}
In this case we call $t_1\cdots t_p$ a {\em generating word} of $d$.
}
\end{definition}

Observe that there are infinitely many admissible words. For instance, when $\alpha=1$, any word of the form $1^k 0^l$ with $k\geq l\geq 1$ is admissible. As a result, there are infinitely many $p$-mirror sequences.

\begin{remark} \label{rem:p-mirror}
{\rm
The term ``$p$-mirror sequence" was introduced by Allouche and Cosnard \cite{Allouche}. Kong and Li \cite{KongLi} use the term ``generalized Thue-Morse sequence", since the special case $\alpha=1$, $p=1$ and $t_1=0$ results in the (truncated) Thue-Morse sequence $110100110010110\cdots$. Note that, in order for this special case to be included, we slighly weakened the definition from \cite{KongLi}, where the right-hand side of the first inequality in \eqref{eq:admissible-block} is given as $t_1\cdots t_p$ (and the inequality is written in its reflected form). It also seems aesthetically pleasing that the two inequalities in \eqref{eq:admissible-block} have the same right-hand side. In any case, the two definitions lead to the same set of $p$-mirror sequences, since it is easy to see that, if $t_1\cdots t_p$ is admissible in the sense of Definition \ref{def:admissible-block}, then the word $s_1\cdots s_{2p}:=t_1\cdots t_p^+\overline{t_1\cdots t_p^+}$ is admissible in the sense of Kong and Li, and generates the same $p$-mirror sequence. This shows in addition that every $p$-mirror sequence has infinitely many generating words.
}
\end{remark}

Kong and Li \cite{KongLi} show that any $p$-mirror sequence $d$ is the quasi-greedy expansion of $1$ in some base $\beta$, and that this $\beta$ is necessarily transcendental. Slightly modifying their terminology, we shall call such a base $\beta$ a {\em de Vries-Komornik number}. It is clear that there are infinitely many such numbers; the smallest is the Komornik-Loreti constant $\beta_c(\alpha)$, which corresponds to the $1$-mirror sequence generated by $t_1=\lfloor \alpha/2 \rfloor$. See \cite{KongLi} for further details.

Observe that any $p$-mirror sequence $d$ satisfies
\begin{equation*}
\sigma^n(d)<d \qquad \mbox{and} \qquad \sigma^n(\bar{d})<d \qquad \mbox{for all $n\in\NN$},
\end{equation*}
so the corresponding bases lie in $\A(\alpha)$ by a result in \cite{KomLor3}. In fact, it follows from \cite[Lemma 4.2]{KomLor3} and its proof that the set of de Vries-Komornik numbers is dense in $\A(\alpha)$. In \cite{deVries}, they are shown to be exactly the right endpoints of the connected components of $(1,\alpha+1]\backslash \overline{\A(\alpha)}$.

The significance of $p$-mirror sequences for the purposes of this article lies in the following.

\begin{lemma} \label{lem:few-choices}
Let $d$ be a $p$-mirror sequence with corresponding de Vries-Komornik number $\beta$. Define $b_m:=d_1\cdots d_{2^{m-1}p}$, $m\in\NN$. Suppose $\omega=\omega_1\omega_2\cdots\in\UU_\beta$. Then for every $m\in\NN$ and $k\geq 0$, the following statements hold:
\begin{enumerate}[(i)]
\item If $\omega_{k+1}\cdots\omega_{k+2^{m-1}p}=b_m$, then $\omega_{k+1}\cdots \omega_{k+2^m p}=b_m\overline{b_m}$ or $b_{m+1}$.
\item If $\omega_{k+1}\cdots\omega_{k+2^{m-1}p}=\overline{b_m}$, then $\omega_{k+1}\cdots \omega_{k+2^m p}=\overline{b_m}b_m$ or $\overline{b_{m+1}}$.
\end{enumerate}
Vice versa, for any sequence $(j_1,j_2,\dots)$ with $0\leq j_\nu<\infty$ for each $\nu$ and $\sum_{\nu=1}^\infty j_\nu=\infty$, the sequence $\omega$ lies in $\UU_\beta$, where
\begin{equation}
\omega=(b_1\overline{b_1})^{j_1}(b_2\overline{b_2})^{j_2}\cdots (b_\nu\overline{b_\nu})^{j_\nu}\cdots.
\label{eq:longer-and-longer}
\end{equation}
\end{lemma}

\begin{proof}
That any $\omega\in\UU_\beta$ satisfies (i) and (ii) was proved in \cite[Lemma 4.11]{KLD} for the case $p=1$. Since the argument for general $p$ is practically the same, we omit it here.

As for the second statement, let $\omega$ be as in \eqref{eq:longer-and-longer}. Given $n\geq 0$, there are integers $k\geq 0$ and $m\in\NN$ such that $k\leq n<k+2^{m-1}p$ and $\omega_{k+1}\cdots \omega_{k+2^{m-1}p}=b_m$ or $\overline{b_m}$. Without loss of generality assume the former, and assume $m$ is the largest integer with this property, so that $\omega_{k+2^{m-1}+1}\cdots\omega_{k+2^m p}=\overline{b_m}$. Since $b_m\overline{b_m}<b_{m+1}$ and these two words differ only in their last digit, it follows that $\sigma^n(\omega)<\sigma^{n-k}(d)\leq d$. Furthermore, it is a consequence of \cite[Lemma 4.2]{KongLi} that $\sigma^n(\bar{\omega})<d$. Hence, by \eqref{eq:U-beta-characterization}, $\omega\in\UU_\beta$.
\end{proof}

It is easy to see that any $p$-mirror sequence $d$ satisfies \eqref{eq:self-reflect}. This leads us to a more precise version of Theorem \ref{thm:Hausdorff-dimension}(iii).

\begin{theorem} \label{thm:Hausdorff-dimension-p-mirror}
If $\beta$ is a de Vries-Komornik number, then $\WW_\beta$ is uncountable but $\dim_H \WW_\beta=0$.
\end{theorem}

The proof of Theorem \ref{thm:Hausdorff-dimension-p-mirror} is given at the end of the paper. It implies the statements of Theorem \ref{thm:Hausdorff-dimension}(iii) since there are infinitely many de Vries-Komornik numbers, and $\beta_c$ is the smallest such number.

\begin{proof}[Proof of Proposition \ref{prop:smallness}]
For $1<\beta<\beta_c$, $\UU_\beta$ contains only eventually periodic sequences \cite[Lemma 4.12]{KLD}, and it is clear that an eventually periodic sequence lies in $\widetilde{\UU}_\beta$ whenever it lies in $\UU_\beta$. Hence, $\WW_\beta=\emptyset$ for all $1<\beta<\beta_c$.

On the other hand, according to \cite[Theorem 2.5]{KongLi}, the interval $(\beta_c,\alpha)$ is covered, with the exception of a Lebesgue-null set, by the collection of intervals $(\beta_L,\beta_R):=(\beta_L(t_1\cdots t_p),\beta_R(t_1\cdots t_p))$, where $t_1\cdots t_p$ ranges over all admissible words (see Definition \ref{def:admissible-block}) of all lengths $p\geq 2$; $d(\beta_L)=(t_1\cdots t_p)^\infty$; and $d(\beta_R)$ is the $p$-mirror sequence with generating word $t_1\cdots t_p$. It is implicit in Theorem 1.7 of \cite{deVries} that any such interval $(\beta_L,\beta_R)$ can be written as a countable union of half open intervals $(\beta_i,\beta_{i+1}]$, $i=1,2,\dots$, so-called {\em stability intervals}, with the property that for each $i$, $\UU_\beta$ is constant on $(\beta_i,\beta_{i+1}]$. Combined with Lemma \ref{lem:containment}, this yields that $\WW_\beta=\emptyset$ for $\beta\in(\beta_i,\beta_{i+1}]$, and hence for $\beta\in(\beta_L,\beta_R)$. As a result, the set $\{\beta\in(1,\alpha+1): \WW_\beta=\emptyset\}$ contains a dense open set of full Lebesgue measure. Both properties stated in the proposition now follow.
\end{proof}

\begin{remark}
{\rm
One can say a bit more about the intervals $(\beta_L,\beta_R)$ in the above proof: By \cite[Theorem 2.6]{KongLi}, the Hausdorff dimension of $\UU_\beta$ with respect to a {\em fixed} metric $\rho_{\beta_0}$ is constant on each such interval, although $\UU_\beta$ itself is not.
}
\end{remark}


\section{Proofs of Theorems \ref{thm:when-non-empty} and \ref{thm:rate-of-convergence}} \label{sec:main-proof}

We begin by introducing the basic building blocks for constructing sequences in $\WW_\beta$.

\begin{definition}
{\rm
Let $d=(d_i)_i\in\Omega$. A sequence $\MM=(m_k)_{k\in\NN}$ of positive integers is {\em $d$-positive} if $d_{m_k}>0$ for every $k\in\NN$.
}
\end{definition}

\begin{definition} \label{def:building-block}
{\rm
For a $d$-positive sequence $\MM=(m_k)_{k\in\NN}$, we define a sequence $\omega(\MM)\in\Omega$ by
\begin{equation*}
\omega(\MM):=\prod_{k=1}^\infty d_1\cdots d_{m_k}^-:=d_1\cdots d_{m_1}^-\,d_1\cdots d_{m_2}^-\cdots.
\end{equation*}
}
\end{definition}

It follows immediately that, if $d$ is the quasi-greedy expansion of $1$ in some base $\beta$, then 
\begin{equation}
\sigma^n(\omega(\MM))<d\quad\mbox{for every $n\geq 0$}. 
\label{eq:small-enough}
\end{equation}
Moreover,
\begin{equation}
\limsup_{n\to\infty} \sigma^n(\omega(\MM))=d \qquad \Longleftrightarrow \qquad \limsup_{k\to\infty} m_k=\infty.
\label{eq:limsup-equivalence}
\end{equation}
Observe also that the mapping $\MM\mapsto \omega(\MM)$ is an injection: If $\MM_1=(m_k^{(1)})$ and $\MM_2=(m_k^{(2)})$ are different $d$-positive sequences, then $\omega(\MM_1)\neq\omega(\MM_2)$.

The critical step in the proofs below is to choose a $d$-positive sequence $\MM=(m_k)$ that satisfies the properties in \eqref{eq:limsup-equivalence} and 
\begin{equation}
\sigma^n\big(\overline{\omega(\MM)}\big)<d \quad\mbox{for every $n\geq 0$},
\label{eq:large-enough}
\end{equation}
so that $\omega(\MM)\in\UU_\beta$ in view of \eqref{eq:U-beta-characterization} and \eqref{eq:small-enough}.
Doing so entails slightly different subtleties in the two cases (i) $\limsup \sigma^n(\bar{d})<d$ and (ii) $\limsup \sigma^n(\bar{d})=d$. We deal with these two cases in Theorems \ref{thm:key-construction} and \ref{thm:critical-case}, respectively.

\begin{theorem} \label{thm:key-construction}
Let $d=d(\beta)$, and suppose
\begin{equation}
\sup_{n\in\NN}\sigma^n(\bar{d})<d. 
\label{eq:subcritical-case}
\end{equation}
Then there is a strictly increasing $d$-positive sequence $\MM_0=(m_k)$ such that, whenever $\MM=(\mu_j)$ is a sequence such that $\mu_j$ is a term of $\MM_0$ for each $j$, we have $\sigma^n\big(\overline{\omega(\MM)}\big)<d$ for all $n\geq 0$.
\end{theorem}

\begin{proof}
Let $\KK$ be the set of those positive integers $k$ for which the sequence $\bar{d}$ contains the word $d_1\cdots d_k$ infinitely many times. By the hypothesis \eqref{eq:subcritical-case}, $\KK$ is finite. Let $k_0:=0$ if $\KK=\emptyset$, and $k_0:=\max\KK$ otherwise. Let $l_0$ be the largest integer $l$ for which $\bar{d}$ contains the word $d_1\cdots d_l$ (or $l_0:=0$ if no such $l$ exists), so $l_0\geq k_0$. If $l_0=k_0$, set $N:=0$; otherwise, set
\begin{equation*}
N_k:=\max\{n\geq 0: \overline{d_{n+1}\cdots d_{n+k}}=d_1\cdots d_k\}+k, \qquad k=k_0+1,\dots,l_0,
\end{equation*}
and
\begin{equation*}
N:=\max\{N_k: k=k_0+1,\dots,l_0\}.
\end{equation*}
Now set $m_0:=N$, and define recursively, for $k=1,2,\dots$,
\begin{align*}
n_k:&=\begin{cases}\min\{n>m_{k-1}: \overline{d_{n+1}\cdots d_{n+k_0}}=d_1\cdots d_{k_0}\} & \mbox{if $k_0>0$},\\
m_{k-1}+1 & \mbox{if $k_0=0$},
\end{cases}\\
m_k:&=n_k+k_0+1.
\end{align*}
It is clear that the sequence $\MM_0:=(m_k)_k$ is strictly increasing. Furthermore, we claim that $\MM_0$ is $d$-positive. To see this, fix $k\in\NN$, and note that on the one hand,
\begin{equation}
\overline{d_{n_k+1}\cdots d_{n_k+k_0}}=d_1\cdots d_{k_0},
\label{eq:match-up-to-here}
\end{equation}
by definition of $n_k$; while on the other hand,
\begin{equation}
\overline{d_{n_k+1}\cdots d_{n_k+k_0+1}}<d_1\cdots d_{k_0+1}.
\label{eq:but-no-farther}
\end{equation}
The last inequality is immediate if $l_0=k_0$, and otherwise follows since $n_k\geq n_1>N\geq N_{k_0+1}$. 
Together, \eqref{eq:match-up-to-here} and \eqref{eq:but-no-farther} yield 
\begin{equation}
\bar{d}_{m_k}<d_{k_0+1}\leq\alpha, 
\label{eq:d-positive-property}
\end{equation}
and as a result, $d_{m_k}>0$.

Now let $\MM=(\mu_j)_j$ be a sequence such that for each $j$, there is a $k$ such that $\mu_j=m_k$. For any such pair $(j,k)$, we also write $\nu_j:=n_k$. Consider the sequence $\omega:=\omega(\MM)$; our goal is to show that $\sigma^n(\bar{\omega})<d$ for all $n\geq 0$. Set $s_0:=0$ and $s_j:=\mu_1+\dots+\mu_j$, $j\geq 1$. Then the sequence $(s_j)$ is strictly increasing, and
\begin{align}
\omega_{s_j}&=d_{\mu_j}^-, \qquad j\in\NN, \notag \\
\omega_{s_j+i}&=d_i, \qquad i=1,2,\dots,\mu_{j+1}-1, \quad j\geq 0. \label{eq:match-here}
\end{align}
Fix $n\geq 0$, and let $j$ be the integer such that $s_j\leq n<s_{j+1}$. We consider two cases.

\bigskip
{\em Case 1.} Suppose $s_j\leq n<s_j+\nu_{j+1}$, or equivalently, $s_j\leq n<s_{j+1}-k_0-1$. Let $l:=s_{j+1}-n-1$. Then $l>k_0$. If in fact $l>l_0$, then we have immediately that
\begin{equation}
\overline{d_{n-s_j+1}\cdots d_{s_{j+1}-s_j-1}}<d_1\cdots d_l,
\label{eq:match-less-than-l}
\end{equation}
by definition of $l_0$. Suppose $k_0<l\leq l_0$. Then 
\begin{align*}
N_l&\leq N<n_1\leq \nu_{j+1}\leq \nu_{j+1}+k_0\\
&=\mu_{j+1}-1=s_{j+1}-s_j-1\\
&=n-s_j+l;
\end{align*}
that is, $N_l<n-s_j+l$. Thus, by definition of $N_l$, \eqref{eq:match-less-than-l} holds also in this case.
But \eqref{eq:match-less-than-l} and \eqref{eq:match-here} together imply
\begin{equation*}
\overline{\omega_{n+1}\cdots\omega_{s_{j+1}-1}}=\overline{d_{n-s_j+1}\cdots d_{s_{j+1}-s_j-1}}<d_1\cdots d_l,
\end{equation*}
and so $\sigma^n(\bar{\omega})<d$.

\bigskip
{\em Case 2.} Suppose $s_j+\nu_{j+1}\leq n<s_{j+1}$. Note that we have
\begin{equation}
\sigma^n(\bar{\omega})=\overline{d_{n-s_j+1}\cdots d_{\mu_{j+1}}}^{\,+}\prod_{r=j+2}^\infty \overline{d_1\cdots d_{\mu_r}}^{\,+}.
\label{eq:long-product-expression}
\end{equation}
We first claim that
\begin{equation}
\overline{d_{k_0+2}\cdots d_{k_0+\mu_{j+2}}}<d_1\cdots d_{\mu_{j+2}-1}.
\label{eq:case2-claim}
\end{equation}
Since $\sigma^{k_0+1}(\bar{d})<d$, we have ``$\leq$". And equality cannot hold by the definition of $l_0$, since
\begin{equation*}
\mu_{j+2}-1\geq m_1-1=n_1+k_0\geq l_0,
\end{equation*}
where the last inequality is obvious if $l_0=k_0$, and follows otherwise since $n_1+k_0\geq n_1>N\geq N_{l_0}\geq l_0$.

Reflecting both sides of \eqref{eq:case2-claim} gives
\begin{equation*}
\overline{d_1\cdots d_{\mu_{j+2}-1}}<d_{k_0+2}\cdots d_{k_0+\mu_{j+2}},
\end{equation*}
and therefore,
\begin{equation}
\prod_{r=j+2}^\infty \overline{d_1\cdots d_{\mu_r}}^{\,+}<d_{k_0+2}d_{k_0+3}\cdots.
\label{eq:tail-comparison}
\end{equation}
On the other hand, with $i:=n-s_j-\nu_{j+1}$, we have
\begin{align*}
\overline{d_{n-s_j+1}\cdots d_{\mu_{j+1}}}^{\,+}
&=\overline{d_{\nu_{j+1}+i+1}\cdots d_{\nu_{j+1}+k_0}d_{\mu_{j+1}}}^{\,+}\\
&=d_{i+1}\cdots d_{k_0}\overline{d_{\mu_{j+1}}}^{\,+}\\
&\leq d_{i+1}\cdots d_{k_0}d_{k_0+1},
\end{align*}
where the second equality follows by definition of $\nu_{j+1}$, and the inequality follows by \eqref{eq:d-positive-property}. Putting this last development together with \eqref{eq:tail-comparison} and substituting in \eqref{eq:long-product-expression} gives
\begin{equation*}
\sigma^n(\bar{\omega})<\sigma^i(d)\leq d.
\end{equation*}
In both cases, the conclusion of the theorem follows.
\end{proof}

\begin{theorem} \label{thm:critical-case}
Let $d=d(\beta)$, and suppose that $\sigma^n(\bar{d})<d$ for every $n\in\NN$, but
\begin{equation}
\limsup_{n\to\infty}\sigma^n(\bar{d})=d. 
\label{eq:critical-case}
\end{equation}
Then there is a strictly increasing $d$-positive sequence $\MM_0=(m_k)$ such that, for any subsequence $\MM=(m_{k(j)})$ of $\MM_0$, we have $\sigma^n\big(\overline{\omega(\MM)}\big)<d$ for all $n\geq 0$.
\end{theorem}

\begin{proof}
We first define the {\em run lengths}
\begin{equation*}
t_n:=\max\{j\in\NN: \overline{d_{n+1}\cdots d_{n+j}}=d_1\cdots d_j\}, \qquad n\in\NN,
\end{equation*}
where we set $\max\emptyset:=0$. Note that \eqref{eq:critical-case} implies $\limsup_{n\to\infty}t_n=\infty$. Set $l_1:=1$, and
recursively, for $k=1,2,\dots$, define:
\begin{gather*}
n_k:=\min\{n\geq 0: \overline{d_{n+1}\cdots d_{n+l_k}}=d_1\cdots d_{l_k}\},\\
m_k:=n_k+\min\{i\in\NN:\bar{d}_{n_k+i}<d_i\},\\
l_{k+1}:=\max\{t_n: n\leq m_{k}\}+1.
\end{gather*}
Then $m_k>n_k>m_{k-1}$, and $\bar{d}_{m_k}<d_{m_k-n_k}$, so $d_{m_k}>0$. Hence $\MM_0:=(m_k)$ is a strictly increasing $d$-positive sequence. For the sake of readability, we will prove the theorem only for the sequence $\MM_0$ itself. The proof for arbitrary subsequences, while notationally more cumbersome, is essentially the same.

Let $\omega:=\omega(\MM_0)$; we will show that $\sigma^n(\bar{\omega})<d$ for all $n\geq 0$. Set $s_0:=0$, and $s_k:=m_1+\dots+m_k$, $k\in\NN$. As in the proof of Theorem \ref{thm:key-construction}, the sequence $(s_k)$ is strictly increasing, and
\begin{align}
\omega_{s_k}&=d_{m_k}^-, \qquad k\in\NN, \notag \\
\omega_{s_k+i}&=d_i, \qquad i=1,2,\dots,m_{k+1}-1, \quad k\geq 0. \label{eq:match-here-again}
\end{align}
Fix $n\geq 0$, and let $k$ be the integer such that $s_{k-1}\leq n<s_{k}$. We consider two cases.

\bigskip
{\em Case 1.} Suppose $s_{k-1}\leq n<s_{k-1}+n_{k}$. We claim that
\begin{equation}
\overline{d_{n-s_{k-1}+1}\cdots d_{m_{k}-1}}<d_1\cdots d_{m_{k}-(n-s_{k-1})-1}.
\label{eq:case1-shorter-match}
\end{equation}
Since $\sigma^{n-s_{k-1}}(\bar{d})<d$ we have ``$\leq$", and equality cannot hold by the choice of $n_{k}$, because $n-s_{k-1}<n_{k}$ and
$m_{k}-(n-s_{k-1})-1>m_{k}-n_{k}-1\geq l_{k}$.

Using \eqref{eq:match-here-again}, \eqref{eq:case1-shorter-match} can be written as
\begin{equation*}
\overline{\omega_{n+1}\cdots\omega_{s_{k}-1}}<d_1\cdots d_{m_{k}-(n-s_{k-1})-1},
\end{equation*}
and so $\sigma^n(\bar{\omega})<d$.

\bigskip
{\em Case 2.} Suppose $s_{k-1}+n_{k}\leq n<s_{k}$. Let $i:=n-s_{k-1}-n_{k}$, so $0\leq i<m_{k}-n_{k}$. We claim first that
\begin{equation}
\overline{d_{m_k-n_k-i+1}\cdots d_{m_k-n_k-i+m_{k+1}-1}}<d_1\cdots d_{m_{k+1}-1}.
\label{eq:case2-shorter-match}
\end{equation}
Since $\sigma^{m_k-n_k-i}(\bar{d})<d$ we have ``$\leq$", and equality cannot hold by the choice of $n_{k+1}$, since $m_{k+1}-1\geq m_{k+1}-n_{k+1}>l_{k+1}$, and $m_k-n_k-i<m_k<n_{k+1}$.

Taking complements in \eqref{eq:case2-shorter-match} gives
\begin{equation*}
\overline{d_1\cdots d_{m_{k+1}-1}}<d_{m_k-n_k-i+1}\cdots d_{m_k-n_k-i+m_{k+1}-1},
\end{equation*}
and hence
\begin{equation}
\prod_{r=k+1}^\infty \overline{d_1\cdots d_{m_r}}^{\,+}<\sigma^{m_k-n_k-i}(d).
\label{eq:this-tail-smaller}
\end{equation}
On the other hand,
\begin{align*}
\overline{d_{n-s_{k-1}+1}\cdots d_{m_k}}^{\,+}&=\overline{d_{n_k+i+1}\cdots d_{m_k}}^{\,+}\\
&=d_{i+1}\cdots d_{m_k-n_k-1}\overline{d_{m_k}}^{\,+}\\
&\leq d_{i+1}\cdots d_{m_k-n_k-1}d_{m_k-n_k}\\
&\leq d_1\cdots d_{m_k-n_k-i},
\end{align*}
where the second equality follows by the definitions of $n_k$ and $m_k$; the first inequality follows since by the choice of $m_k$, $\bar{d}_{m_k}<d_{m_k-n_k}$; and the last inequality follows since $\sigma^i(d)\leq d$. Combining this last development with \eqref{eq:this-tail-smaller}, we obtain (compare with \eqref{eq:long-product-expression}) 
\begin{align*}
\sigma^n(\bar{\omega})&=\overline{d_{n-s_{k-1}+1}\cdots d_{m_k}}^{\,+}\prod_{r=k+1}^\infty \overline{d_1\cdots d_{m_r}}^{\,+}\\
&<d_1\cdots d_{m_k-n_k-i} \sigma^{m_k-n_k-i}(d)=d.
\end{align*}
This completes the proof.
\end{proof}

\begin{proof}[Proof of Theorem \ref{thm:when-non-empty}]
(i) $\Rightarrow$ (iii): If $\sigma^{n_0}(\bar{d})\geq d$ for some $n_0$, then $\WW_\beta=\emptyset$. To see this, let $\omega\in\UU_\beta$. We claim that
\begin{equation}
\omega_{n+1}\cdots \omega_{n+n_0}<d_1\cdots d_{n_0} \qquad\mbox{for all $n\geq 0$}.
\label{eq:cannot-follow}
\end{equation}
Let $n\geq 0$. Since $\omega\in\UU_\beta$ we have $\sigma^{n+n_0}(\bar{\omega})<d$, and so $\sigma^{n_0}(\bar{d})\geq d>\sigma^{n+n_0}(\bar{\omega})$. Taking complements this gives $\sigma^{n_0}(d)<\sigma^{n+n_0}(\omega)$. Writing
\begin{equation*}
\sigma^n(\omega)=\omega_{n+1}\cdots\omega_{n+n_0}\sigma^{n+n_0}(\omega), \qquad d=d_1\cdots d_{n_0}\sigma^{n_0}(d),
\end{equation*}
and using that $\sigma^n(\omega)<d$ since $\omega\in\UU_\beta$, \eqref{eq:cannot-follow} follows. But \eqref{eq:cannot-follow} in turn implies that $\limsup \sigma^n(\omega)<d$. Interchanging the roles of $\omega$ and $\bar{\omega}$ (and using that $\omega\in\UU_\beta$ if and only if $\bar{\omega}\in\UU_\beta$), the same argument gives $\limsup \sigma^n(\bar{\omega})<d$. Thus, by Proposition \ref{prop:translation} and \eqref{eq:alternative-representation-of-U-tilde}, $\omega\in\widetilde{\UU}_\beta$.

(iii) $\Rightarrow$ (ii): Since any strictly increasing sequence $(m_k)$ has uncountably many different subsequences, this implication follows from the injectivity of the map $\MM\mapsto\omega(\MM)$, together with Theorem \ref{thm:key-construction} in case $\limsup \sigma^n(\bar{d})<d$, or Theorem \ref{thm:critical-case} in case $\limsup \sigma^n(\bar{d})=d$.

Finally, (ii) $\Rightarrow$ (i) is obvious, and the equivalence (iii) $\Leftrightarrow$ (iv) was proved in \cite{deVries}, as pointed out earlier.
\end{proof}


\begin{proof}[Proof of Theorem \ref{thm:rate-of-convergence}]
(i) Assume that $\WW_\beta\neq\emptyset$, and let a sequence $(\theta_n)$ of positive numbers be given. Let $\MM_0=(m_k)$ be the sequence from Theorem \ref{thm:key-construction} or Theorem \ref{thm:critical-case}, as appropriate, and extract from $\MM_0$ a subsequence $\MM:=(m_{k(j)})$ which grows sufficiently fast so that
\begin{equation}
\beta^{-m_{k(j+1)}}\leq \theta_{s_j}^{-1} \qquad\mbox{for every $j\in\NN$},
\label{eq:grow-fast-enough}
\end{equation}
where $s_j:=m_{k(1)}+\dots+m_{k(j)}$. This can always be done, since the sequence $(m_k)$ increases to $+\infty$. Let $\omega:=\omega(\MM)$. Theorem \ref{thm:key-construction} or \ref{thm:critical-case} guarantees $\omega\in\UU_\beta$. By Lemma \ref{lem:bi-Lipschitz}, there is a constant $K>0$ such that
\begin{equation*}
|1-\Pi_\beta(\sigma^n(\omega))|=|\Pi_\beta(d)-\Pi_\beta(\sigma^n(\omega))|\leq K\rho_\beta(d,\sigma^n(\omega)), \qquad n\in\NN.
\end{equation*}
Taking $n=s_j$ and using \eqref{eq:grow-fast-enough} and the definition of the metric $\rho_\beta$, we obtain
\begin{equation*}
1-\Pi_\beta(\sigma^{s_j}(\omega))\leq K\beta^{-m_{k(j+1)}}\leq K\theta_{s_j}^{-1}, \qquad j\in\NN,
\end{equation*}
and hence, 
\begin{equation}
\liminf_{n\to\infty} \theta_n\big(1-\Pi_\beta(\sigma^n(\omega))\big)\leq K<\infty.
\label{eq:converge-fast-enough}
\end{equation}
Finally, since there are clearly uncountably many subsequences $(m_{k(j)})$ satisfying \eqref{eq:grow-fast-enough} and the map $\MM\mapsto \omega(\MM)$ is injective, there are uncountably may sequences $\omega\in\WW_\beta$ satisfying \eqref{eq:converge-fast-enough}.

(ii) Since \eqref{eq:both-infinite} is satisfied for all $\omega\in\widetilde{\UU}_\beta$ and $\UU_\beta$ is uncountable for every $\beta>\beta_c$, the second statement follows immediately from Lemma \ref{lem:containment} for $\beta>\beta_c$. Assume $\beta=\beta_c$, and recall that $d:=d(\beta)$ is a $p$-mirror sequence with $p=1$. Let $b_i:=d_1\cdots d_{2^{i-1}}$, $i\in\NN$. Choose a sequence $(k_i)_{i\in\NN}$ so that 
\begin{equation}
\theta_{k_i}\geq\beta^{2^{i+1}}, \qquad i\in\NN,
\label{eq:super-fast-growth}
\end{equation}
and let
\begin{equation}
\omega=(b_1\overline{b_1})^{k_1}(b_2\overline{b_2})^{k_2}(b_3\overline{b_3})^{k_3}\cdots.
\label{eq:repeat-often-enough}
\end{equation}
By the second part of Lemma \ref{lem:few-choices}, $\omega\in\UU_\beta$. For any finite word $u$, let $|u|$ denote the length (number of symbols) of $u$. Then $|b_i|=2^{i-1}$ and so $|(b_i\overline{b_i})^{k_i}|=k_i 2^i$. Given $n\in\NN$, let $j$ be the integer such that
$\sum_{i=1}^j k_i 2^i\leq n<\sum_{i=1}^{j+1} k_i 2^i$.
Then $\sigma^n(\omega)$ coincides with $d$ for at most the first $|b_{j+1}|$ digits, and the same is true for $\sigma^n(\bar{\omega})$. Thus, $\rho_\beta(\sigma^n(\omega),d)\geq \beta^{-2^j}$, and it follows from Lemma \ref{lem:bi-Lipschitz} that there is a constant $C>0$ for which
\begin{equation*}
1-\Pi_\beta(\sigma^n(\omega))=|\Pi_\beta(d)-\Pi_\beta(\sigma^n(\omega))|\geq C\beta^{-2^j}, \qquad j\in\NN,
\end{equation*}
and the same with $\bar{\omega}$ in place of $\omega$. On the other hand, $n>k_j$, so $\theta_n\geq\theta_{k_j}$ and \eqref{eq:super-fast-growth} implies
\begin{equation*}
\theta_n\big(1-\Pi_\beta(\sigma^n(\omega))\big)\geq C\beta^{2^j}\to\infty,
\end{equation*}
and the same for $\bar{\omega}$. Finally, since for each $i$ we could replace $k_i$ with $k_i+1$ and $b_i\overline{b_i}\neq b_{i+1}$, there are uncountably many $\omega$ of the form \eqref{eq:repeat-often-enough} with this property.
\end{proof}

\section{Proofs of Theorems \ref{thm:Hausdorff-dimension} and \ref{thm:Hausdorff-dimension-p-mirror}} \label{sec:HD}

Instead of proving Theorem \ref{thm:Hausdorff-dimension} directly, it is notationally simpler to prove a more general and slightly stronger result. To this end, we slightly extend the notation from earlier sections. Let $A$ be any finite alphabet, write $A^*:=\bigcup_{n=1}^\infty A^n$, and let $\Omega:=A^\NN$. Let $0<\lambda<1$ be a constant, and equip $\Omega$ with the metric $\rho(\omega,\eta)=\lambda^{\min\{n:\omega_n\neq\eta_n\}-1}$, for $\omega=\omega_1\omega_2\cdots$ and $\eta=\eta_1\eta_2\cdots$. For any subset $V$ of $A^*$, let $V^\NN$ denote the set of all infinite concatenations $v_1 v_2 v_3\cdots$, where $v_i\in V$ for each $i$. For $v\in A^*$, let $|v|$ denote the length of $v$; that is, $|v|=n$ when $v\in A^n$. If $u,v\in A^*$ and $|u|<|v|$, we say $v$ {\em extends} $u$ if $v_1\cdots v_{|u|}=u$. The following fact is well known; see, for instance, \cite[Proposition 3]{Staiger}.

\begin{theorem} \label{thm:infinite-IFS}
Let $V\subset A^*$ such that no word in $V$ extends any other word in $V$. Then $\dim_H V^\NN$ is the unique real number $s$ such that
\begin{equation*}
\sum_{v\in V}\lambda^{s|v|}=1.
\end{equation*}
\end{theorem}

Now let $V=\{v_1,v_2,\dots\}$ be an infinite subset of $A^*$ such that no word in $V$ extends any other. Let
\begin{equation*}
E:=V^\NN\backslash \bigcup_{i=1}^\infty (V\backslash\{v_i\})^\NN.
\end{equation*}
In other words, $E$ is the set of all infinite concatenations $v_{i_1}v_{i_2}\cdots$ of words from $V$ with the property that each $i\in\NN$ occurs at least once in the sequence $(i_1,i_2,\dots)$. The following theorem may well be known, but since the author could not find it in the literature, a proof is included for completeness.

\begin{theorem} \label{thm:general-Hausdorff-dimension}
It holds that $\dim_H E=\dim_H V^\NN$.
\end{theorem}

\begin{proof}
Assume the elements of $V$ are ordered such that $|v_1|\leq |v_2|\leq |v_3|\leq \dots$.
Since Theorem \ref{thm:infinite-IFS} easily implies that $\dim_H V^\NN=\sup\{\dim_H W^\NN: W\subset V, \#W<\infty\}$, it suffices to show that $\dim_H E\geq \dim_H W^\NN$ for any finite subset $W$ of $V$. Fix such a set $W$, and let $s:=\dim_H W^\NN$, so that
\begin{equation}
\sum_{v\in W}\lambda^{s|v|}=1.
\label{eq:Moran-equation}
\end{equation}
Fix $0<t<s$, and let $K=(k_l)_{l\in\NN}$ be a strictly increasing sequence of positive integers such that
\begin{equation}
k_l\geq l\left(1+\frac{t|v_l|}{(s-t)|v_1|}\right), \qquad l\in\NN.
\label{eq:K-growth-rate}
\end{equation}
Slightly abusing notation, define, for $k\in\NN$, the index set
\begin{equation*}
I_k:=\begin{cases}
\{i\in\NN: v_i\in W\}, & \mbox{if $k\not\in K$},\\
\{l\}, & \mbox{if $k=k_l\in K$}.
\end{cases}
\end{equation*}
Define the set
\begin{equation*}
E_{W,K}:=\{v_{i_1}v_{i_2}v_{i_3}\cdots: i_k\in I_k\ \mbox{for all $k\in\NN$}\}.
\end{equation*}
Clearly $E\supset E_{W,K}$. We will show that $\dim_H E_{W,K}\geq t$.

For a word $v\in A^*$, define the cylinder set $[v]$ by $[v]:=v\cdot A^\NN$, the set of all sequences in $\Omega$ that begin with $v$. Write
\begin{equation*}
C_{i_1\cdots i_n}:=[v_{i_1}\cdots v_{i_n}], \qquad i_1,\dots,i_n\in\NN.
\end{equation*}
Observe that the above cylinder set has diameter
\begin{equation*}
|C_{i_1\cdots i_n}|=\lambda^{|v_{i_1}|+\dots+|v_{i_n}|}.
\end{equation*}
Put 
\begin{equation*}
I_{W,K}^n:=\bigotimes_{k=1}^n I_k \quad (n\in\NN), \qquad I_{W,K}^*:=\bigcup_{n=1}^\infty I_{W,K}^n.
\end{equation*}
Further, denote
\begin{equation*}
S_n':=\sum_{l\in\NN:k_l\leq n} |v_l|, \qquad N_n:=\#\{1\leq k\leq n: k\not\in K\}, \qquad \mbox{for $n\in\NN$}.
\end{equation*}
We will define a mass distribution $\mu$ on $E_{W,K}$ as follows. For $(i_1,\dots,i_n)\in \NN^n$, set
\begin{equation*}
\mu(C_{i_1\cdots i_n}):=\begin{cases}
|C_{i_1\cdots i_n}|^s \lambda^{-sS_n'} & \mbox{if $(i_1,\dots,i_n)\in I_{W,K}^*$},\\
0 & \mbox{otherwise}.
\end{cases}
\end{equation*}
If $[\omega_1\cdots\omega_m]$ is an arbitrary cylinder set in $\Omega$ which is not of the form $C_{i_1\cdots i_n}$, let $n$ be the largest integer such that $[\omega_1\cdots\omega_m]\subset C_{i_1\cdots i_n}$ for some $i_1,\dots,i_n$, and set
\begin{equation}
\mu([\omega_1\cdots\omega_m]):=\sum\{\mu(C_{i_1\cdots i_n j}): C_{i_1\cdots i_n j}\subset [\omega_1\cdots\omega_m]\},
\label{eq:in-between-cylinders}
\end{equation}
where the summation is over all $j\in\NN$ satisfying the given condition.
Or, if $\omega_1\cdots \omega_m$ cannot be extended to a sequence in $E_{W,K}$, set $\mu([\omega_1\cdots\omega_m])=0$.
We first check the consistency condition
\begin{equation}
\sum_{j=1}^\infty \mu(C_{i_1\cdots i_n j})=\mu(C_{i_1\cdots i_n}), \qquad n\in\NN, \quad i_1,\dots,i_n\in\NN.
\label{eq:consistency}
\end{equation}
Assume $(i_1,\dots,i_n)\in I_{W,K}^*$, as otherwise \eqref{eq:consistency} is trivial.
If $n+1=k_l\in K$, then $S_{n+1}'=S_n'+|v_l|$, and so
\begin{align*}
\mu(C_{i_1\cdots i_n l})&=|C_{i_1\cdots i_n l}|^s \lambda^{-sS_{n+1}'}
=\left(|C_{i_1\cdots i_n}|\lambda^{|v_l|}\right)^s \lambda^{-sS_n'-s|v_l|}\\
&=|C_{i_1\cdots i_n}|^s \lambda^{-sS_n'}
=\mu(C_{i_1\cdots i_n}),
\end{align*}
while for $j\neq l$, $\mu(C_{i_1\cdots i_n j})=0$. Thus, \eqref{eq:consistency} holds when $n+1\in K$. If $n+1\not\in K$, then $S_{n+1}'=S_n'$, so for $j$ such that $v_j\in W$, we have
\begin{equation*}
\mu(C_{i_1\cdots i_n j})=|C_{i_1 \cdots i_n j}|^s \lambda^{-sS_{n+1}'}
=|C_{i_1\cdots i_n}|^s \lambda^{s|v_j|} \lambda^{-sS_n'}
=\mu(C_{i_1\cdots i_n}) \lambda^{s|v_j|},
\end{equation*}
and in this case, \eqref{eq:consistency} follows from \eqref{eq:Moran-equation}.

Now \eqref{eq:consistency}, together with \eqref{eq:in-between-cylinders} and Kolmogorov's consistency theorem, implies that $\mu$ can be extended to a unique Borel probability measure on $\Omega$, and $\mu(E_{W,K})=1$. We claim that
\begin{equation}
\mu(C_{i_1\cdots i_n})\leq |C_{i_1\cdots i_n}|^t,
\label{eq:measure-estimate}
\end{equation}
for all $(i_1,\dots,i_n)\in I_{W,K}^*$. Fix such an $n$-tuple $(i_1,\dots,i_n)$, and set $S_n:=|v_{i_1}|+\dots+|v_{i_n}|$. Observe that
\begin{equation}
S_n\geq S_n'+N_n|v_1|.
\label{eq:sum-lower-estimate}
\end{equation}
Let $l\in\NN$ be such that $k_l\leq n<k_{l+1}$. Then \eqref{eq:K-growth-rate} implies
\begin{equation*}
(s-t)(n-l)|v_1|\geq (s-t)(k_l-l)|v_1|\geq tl|v_l|.
\end{equation*}
Since $S_n'=S_{k_l}'\leq l|v_l|$ and $N_n=n-l$, this gives
\begin{equation*}
(s-t)N_n|v_1|=(s-t)(n-l)|v_1|\geq tl|v_l|\geq tS_n', 
\end{equation*}
which by \eqref{eq:sum-lower-estimate} implies
$(s-t)S_n\geq sS_n'$.
Hence,
\begin{equation*}
|C_{i_1\cdots i_n}|^{s-t}=\lambda^{(s-t)S_n}\leq \lambda^{sS_n'},
\end{equation*}
which yields \eqref{eq:measure-estimate}.

Now let $[\omega_1\cdots\omega_m]$ be an arbitrary cylinder set in $\Omega$. We may assume this set intersects $E_{W,K}$, for otherwise its $\mu$-measure is zero. Let $n$ be the largest integer such that $[\omega_1\cdots\omega_m]\subset C_{i_1\cdots i_n}$ for some $(i_1,\dots,i_n)\in I_{W,K}^n$. If $n+1\not\in K$, then 
\begin{equation*}
\mu([\omega_1\cdots\omega_m])\leq \mu(C_{i_1\cdots i_n})
\leq |C_{i_1\cdots i_n}|^t
\leq\lambda^{-tL}\big|[\omega_1\cdots\omega_m]\big|^t,
\end{equation*}
where $L:=\max\{|v|:v\in W\}$. On the other hand, if $n+1=k_l\in K$, then
\begin{equation*}
\mu([\omega_1\cdots\omega_m])=\mu(C_{i_1\cdots i_n l})
\leq |C_{i_1\cdots i_n l}|^t
\leq \big|[\omega_1\cdots\omega_m]\big|^t.
\end{equation*}
In both cases,
\begin{equation*}
\mu([\omega_1\cdots\omega_m])\leq \lambda^{-tL}\big|[\omega_1\cdots\omega_m]\big|^t=\lambda^{-tL}\lambda^{mt}.
\end{equation*}
Finally, let $U$ be any subset of $\Omega$. Let $m$ be the integer such that $\lambda^m<|U|\leq \lambda^{m-1}$. Then $U$ intersects at most $\#A$ cylinders $[\omega_1\cdots\omega_m]$, and so
\begin{equation*}
\mu(U)\leq (\#A)\lambda^{-tL}\lambda^{mt}<(\#A)\lambda^{-tL}|U|^t.
\end{equation*}
Hence, by the distribution of mass principle (see \cite{Falconer}), $\dim_H E_{W,K}\geq t$. But then also $\dim_H E\geq t$, and letting now $t\uparrow s$ yields $\dim_H E\geq s=\dim_H W^\NN$, as desired.
\end{proof}

It is worth noting that the analog of Theorem \ref{thm:general-Hausdorff-dimension} for {\em finite} $V$ is nearly trivial. To see why, let $s=\dim_H V^\NN$, and let $\mathcal{H}^s$ denote $s$-dimensional Hausdorff measure on $\Omega$. To avoid the trivial case where $\dim_H E=\dim_H V^\NN=0$, assume $\#V\geq 2$. Then $0<\mathcal{H}^s(V^\NN)<\infty$, since $V^\NN$ is the attractor of a finite iterated function system. Write $V=\{v_1,\dots,v_n\}$, and let $F_i:=(V\backslash \{v_i\})^\NN$, $i=1,\dots,n$. Then $\dim_H F_i<s$ and so $\mathcal{H}^s(F_i)=0$, for each $i$, and hence $\mathcal{H}^s(V^\NN\backslash E)=\mathcal{H}^s(\bigcup_{i=1}^n F_i)=0$. But this implies $\mathcal{H}^s(E)>0$, so that $\dim_H E\geq s$. I thank Kenneth Falconer for pointing this out. The above method does not work for infinite $V$, since in general it is possible that $\mathcal{H}^s(V^\NN)=0$. Various sufficient conditions are known under which the attractor of an infinite iterated function system has positive Hausdorff measure in its dimension; see, for instance, Mauldin and Urbanski \cite{MU} or Staiger \cite{Staiger}. But our theorem above holds equally whether or not this is the case.

\begin{proof}[Proof of Theorem \ref{thm:Hausdorff-dimension}]
Statement (i) is immediate from Lemma \ref{lem:containment} and the fact, proved recently by Komornik et al. \cite{KKL}, that the function $\beta\mapsto \dim_H \UU_\beta$ is continuous. For (ii), let $\MM_0=(m_k)$ be the sequence used in the proof of Theorem \ref{thm:when-non-empty}, and apply Theorem \ref{thm:general-Hausdorff-dimension} to the set $V=\{v_1,v_2,\dots\}$, where $v_k=d_1\cdots d_{m_k}^-$ for $k\in\NN$, and $\lambda:=1/\beta$. Statement (iii) follows from Theorem \ref{thm:Hausdorff-dimension-p-mirror}, which is proved below.
\end{proof}

In preparation for the proof of Theorem \ref{thm:Hausdorff-dimension-p-mirror}, we need the following lemma. 

\begin{lemma} \label{lem:binary-tree}
Let $\{a_{i_1,\dots,i_n}: n\in\NN,\ i_1,i_2,\dots\in\{1,2\}\}$ be a binary tree of positive integers defined by $a_1=1$, $a_2=2$, and for $n\in\NN$,
\begin{equation*}
a_{i_1,\dots,i_n,j}=ja_{i_1,\dots,i_n}, \qquad n\in\NN, \quad i_1,\dots,i_n,j\in\{1,2\}.
\end{equation*}
Let $\sigma_{i_1,\dots,i_n}:=1+a_{i_1}+a_{i_1,i_2}+\dots+a_{i_1,i_2,\dots,i_n}$. Let $0<\gamma<1$ be a constant, and define the sum
\begin{equation*}
S_n:=\sum_{i_1,\dots,i_n\in\{1,2\}} \gamma^{\sigma_{i_1,\dots,i_n}}.
\end{equation*}
Then $S_n\to 0$ as $n\to\infty$.
\end{lemma}

\begin{proof}
Define the sums
\begin{equation*}
S_n^L:=\sum_{i_2,\dots,i_n\in\{1,2\}}\gamma^{\sigma_{1,i_2,\dots,i_n}}, \qquad S_n^R:=\sum_{i_2,\dots,i_n\in\{1,2\}}\gamma^{\sigma_{2,i_2,\dots,i_n}},
\end{equation*}
so that $S_n=S_n^L+S_n^R$. Since $a_{1,i_2,\dots,i_n}=a_{i_2,\dots,i_n}$, we have $\sigma_{1,i_2,\dots,i_n}=1+\sigma_{i_2,\dots,i_n}$ and this shows that
\begin{equation}
S_{n+1}^L=\gamma S_n. 
\label{eq:left-sum}
\end{equation}
Furthermore, $a_{2,i_2,\dots,i_n}=2a_{1,i_2,\dots,i_n}$, so $\sigma_{2,i_2,\dots,i_n}=2\sigma_{1,i_2,\dots,i_n}-1$. Clearly $\sigma_{1,i_2,\dots,i_n}\geq n+1$, and so $\sigma_{2,i_2,\dots,i_n}\geq \sigma_{1,i_2,\dots,i_n}+n$. Hence,
\begin{equation}
S_n^R\leq \gamma^n S_n^L.
\label{eq:right-sum}
\end{equation}
From \eqref{eq:left-sum} and \eqref{eq:right-sum} (with $n+1$ in place of $n$ in the latter), it follows that
\begin{equation*}
S_{n+1}=S_{n+1}^L+S_{n+1}^R\leq (1+\gamma^{n+1})S_{n+1}^L
=\gamma(1+\gamma^{n+1})S_n=(\gamma+\gamma^{n+2})S_n.
\end{equation*}
Choose $n_0$ large enough so that $\delta:=\gamma+\gamma^{n_0+2}<1$. Then for each $n\geq n_0$, $S_{n+1}\leq\delta S_n$, from which the result follows.
\end{proof}

\begin{proof}[Proof of Theorem \ref{thm:Hausdorff-dimension-p-mirror}]
Let $\beta$ be a de Vries-Komornik number, and $d:=d(\beta)$.
If $\omega\in\WW_\beta$, then $\limsup_{n\to\infty} \sigma^n(\omega)=d$, so certainly there exists $k\in\NN$ such that $\omega_{k+1}\cdots\omega_{k+p}=d_1\cdots d_p$. Thus, $\WW_\beta$ is covered by countably many similar copies of the set
\begin{equation*}
\mathcal{Y}_\beta:=\{\omega=\omega_1\omega_2\cdots\in\UU_\beta: \omega_1\cdots \omega_p=d_1\cdots d_p\},
\end{equation*}
and it suffices to show that $\dim_H \mathcal{Y}_\beta=0$.

Let $b_m:=d_1\cdots d_{2^{m-1}p}$, $m\in\NN$. Define numbers $\sigma_{i_1,\dots,i_n}$ for $n\in\NN$ and $i_1,\dots,i_n\in\{1,2\}$ as in Lemma \ref{lem:binary-tree}. By Lemma \ref{lem:few-choices}, $\YY_\beta$ is covered at ``level 1" by the two cylinders $C_1:=[b_1\overline{b_1}]$ and $C_2:=[b_2]$. Applying Lemma \ref{lem:few-choices} again, we see that at ``level 2" $\YY_\beta$ is covered by the four cylinders $C_{11}:=[b_1\overline{b_1}b_1]$, $C_{12}:=[b_1\overline{b_2}]$, $C_{21}:=[b_2\overline{b_2}]$ and $C_{22}:=[b_3]$. Continuing this way, we find that at ``level $n$", $\YY_\beta$ is covered by $2^n$ cylinders $C_{i_1,\dots,i_n}$, $i_1,\dots,i_n\in\{1,2\}$, where $C_{i_1,\dots,i_n}$ has depth $p(1+\sigma_{i_1,\dots,i_{n-1}})$, so that
\begin{equation*}
|C_{i_1,\dots,i_n}|=\beta^{-p(1+\sigma_{i_1,\dots,i_{n-1}})}.
\end{equation*}
Now let $s>0$ be given, and apply Lemma \ref{lem:binary-tree} with $\gamma:=\beta^{-ps}$ to obtain
\begin{equation*}
\sum_{i_1,\dots,i_n\in\{1,2\}}|C_{i_1,\dots,i_n}|^s=2\gamma\sum_{i_1,\dots,i_{n-1}\in\{1,2\}}\gamma^{\sigma_{i_1,\dots,i_{n-1}}}\to 0,
\end{equation*}
as $n\to\infty$. Thus, $\dim_H \mathcal{Y}_\beta=0$, as desired.
\end{proof}

\begin{remark}
{\rm
({\em a}) A slight extension of the above argument shows also that $\dim_H \UU_{\beta_c}=0$, since $d:=d(\beta_c)$ is a $p$-mirror sequence. This well-known result was stated for the case $\alpha=1$ in \cite{GlenSid} without a detailed proof, and later for the general case in \cite{KLD}, where Lemma \ref{lem:few-choices} is used but the further details are omitted.

({\em b}) It would be of interest to determine the correct gauge function $h$ for which $0<\mathcal{H}^h(\UU_{\beta_c})<\infty$, where $\mathcal{H}^h$ is the generalized Hausdorff measure induced by $h$. It is clear that an answer to this question would require a substantial refinement of the analysis carried out in the proof of Lemma \ref{lem:binary-tree}.
}
\end{remark}

\section*{Acknowledgment}
The author is grateful to Boris Solomyak for supplying a helpful example that indirectly inspired this work; to Derong Kong for a useful email discussion; and to the referee for many helpful suggestions for improving the presentation of the paper.

\end{document}